%% file: main.tex
\documentclass{amsart}

\usepackage{amsmath}
\usepackage{amsfonts}
\usepackage{amssymb}
\usepackage{amsxtra}
\usepackage{amsthm}
\usepackage{mathrsfs}
\usepackage{mathtools}
\usepackage{bm}
\usepackage{colonequals}
\usepackage{amscd}
\usepackage{tikz}
\usepackage{tikz-cd}

\usepackage[skip=0.4\baselineskip]{parskip}

\usepackage{enumitem}
\usepackage{booktabs}
\usepackage{etoolbox}
\usepackage{multicol}
\usepackage{appendix}

\usepackage{microtype}
\usepackage{url}
\usepackage[hidelinks,pdftex]{hyperref}
\usepackage[capitalise]{cleveref}
\usepackage{cite}

\usepackage{mystyle}
\usepackage[T1]{fontenc}

\title%
[A Ghost Lemma for Commutative Ring Homomorphisms]%
{A Ghost Lemma for Commutative Ring Homomorphisms via Andr\'{e}-Quillen Homology}

\date{\today}

\author[D.~McCormick]{Daniel McCormick}
\address{Department of Mathematics,
  Stockholms Universitet, Stockholm, Sweden}
\email{daniel.mccormick@math.su.se}

\thanks{The author was partly supported by the National Science
  Foundation under grant No. 2001368 and No. 1840190, and the Swedish
  Research Council through grant no. 2021-03946
}



\begin{document}
\input{abstract}
\maketitle
\input{intro}
\input{cats}
\input{aq}
\input{frob_koszul}
\input{ghost_maps}
\input{ghost_lemma}
\input{singularities}

\bibliographystyle{amsplain}
\bibliography{refs.bib}
\end{document}

%% file: abstract.tex
\begin{abstract}
  We adapt the theory of ghost maps from derived categories to the
  setting of commutative rings using \andq\ homology. The Frobenius
  endomorphism is a primary example of a ghost map in this setting. We
  prove an analogue of the ghost lemma for rings and demonstrate its
  utility by deducing a characteristic independent generalization of
  Kunz's theorem and an analogue for complete intersection rings.
\end{abstract}

%% file: intro.tex
\section{Introduction}
We often study objects, such as chain complexes or topological spaces,
by studying their homology. While an object with trivial homology is
homologically trivial, a map between objects may be trivial on
homology even while carrying nontrivial homological information. These
are known as ghost maps and are the subject of this
paper. Ghost maps have been studied in the derived categories of
modules, the stable homotopy category of spaces, as well as in general
triangulated categories \cite{Christensen98, Beligiannis08, Kelly65,
  Letz21, LiuPollitz23, OppermannStovicek12}. The goal of this paper
is to introduce the study of ghost maps to commutative rings.

The derived category of modules is constructed from chain complexes in
order to apply the ideas of homotopy theory to the study of
modules. In this setting, a map of chain complexes \(f \co X \to Y\)
is ghost if \(\h_*(f) = 0\). In a similar fashion, following the work
of Quillen, we use simplicial rings to study the homotopy
theory of rings, providing us with analogues for many of the familiar
constructions and ideas from the derived category, such as free
resolutions, derived functors, etc.

The natural theory of homology for commutative rings, introduced
independently by Andr\'{e} \cite{Andre74} and Quillen
\cite{QuillenHCR,Quillen70}, is known as \andq\ homology and denoted
\(\AQ\). Under the conditions discussed in \cref{sec:aq} we
establish the following definition: if \(f \co R \to S\) is a map
of (simplicial) rings, then \(f\) is \emph{ghost} if \(\AQ_*(f) =
0\). Interpreted appropriately, the Frobenius endomorphism,
\(F(x) = x^p\), defined for any ring of prime characteristic \(p\),
gives a quintessential example of a ghost map.

Within the study of ghost maps, there is a family of important
technical results known as ghost lemmas. If \(R\) is a ring and
\(f \co X \to Y\) is a map of chain complexes of \(R\)-modules, we
have the following prototypical example, originally due to Kelly
\cite{Kelly65} (see also \cite[Theorem 8.3]{Christensen98}).
\begin{quote}
  \textit{Let \(X\) be a complex of \(R\)-modules such that
    \(\h_i(X)\) and \(\mathsf{B}_i(X)\) have projective dimension less
    than \(n\) for each \(i\). If \(f \co X \to Y\) is a composite
    of at least \(n\) ghost maps, then \(f\) is zero up to homotopy.}
\end{quote}
By analogy, we prove the following ghost lemma for rings.
\begin{introtheorem}\label[introtheorem]{thm:intro-ghost-ring}
  Suppose \(k\) is a field and \(S\) is an augmented connected
  simplicial \(k\)-algebra such that the \(i\)\tth\ homotopy group
  \(\pi_i(S)\) vanishes for all \(i \ge 2^n\). If \(f \co R \to S\) is
  a composite of at least \(n\) ghost maps, then \(f\) is trivial up
  to homotopy.
\end{introtheorem}\noindent
In this context, a map of \(k\)-algebras is trivial if it factors
through \(k\).  The homological information encoded by trivial ring
maps has long been used to study singularities, dating at least back
to the Auslander-Buchsbaum-Serre theorem, which can be phrased as
follows:
\begin{quote}
  \textit{Suppose \(R\) is a local ring with residue field \(k\). Then
    \(R\) is regular if and only if any/all trivial ring maps
    \(R \to S\) have finite flat dimension.}
\end{quote}
By contrast, Kunz's theorem characterizes regularity via the
homological properties of the Frobenius endomorphism:
\begin{quote}
  \textit{Suppose \(R\) is a ring with prime characteristic. Then
    \(R\) is regular if and only if the Frobenius endomorphism has
    finite flat dimension.}
\end{quote}
However, the Frobenius endomorphism is ghost. Passing to the
simplicial Koszul complex gives us a ghost map on a connected
simplicial \(k\)-algebra with bounded homotopy groups, and so we use
\cref{thm:intro-ghost-ring} to reduce Kunz's theorem to a case of the
Auslander-Buchsbaum-Serre theorem. Generalizing the proof gives us the
following theorem.
\begin{introtheorem}\label[introtheorem]{thm:intro-kunzy}
  Let \(R\) be a local ring and \(\phi\) a local endomorphism such
  that the induced map on the simplicial Koszul complex is ghost. If
  \(\phi\) has finite flat dimension (e.g. \(\phi\) is flat), then
  \(R\) is regular.
\end{introtheorem}
Further, using similar methods, we prove an analogous result for
complete intersections.
\begin{introtheorem}\label[introtheorem]{thm:intro-kunzy-ci}
  Let \(R\) be a local ring and \(\phi\) a local endomorphism such
  that the induced map on the simplicial Koszul complex is ghost. If
  \(\phi^n\) has finite CI-dimension for some \(n \gg 0\), then \(R\) is
  complete intersection.
\end{introtheorem}

The ghost lemma, \cref{thm:intro-ghost-ring}, can be thought of as a
strengthening of \cite[Prop 3.7]{BILMP2023} when applied to the
Frobenius endomorphism and Koszul complex, see
\cref{rem:frob-is-koszul}, \cref{rem:compare-to-bilmp}, and
\cref{lem:ghost-koszul-trivial}.

\cref{thm:intro-kunzy} can also be proven using the work
of Majadas \cite{Majadas16} and \cref{thm:intro-kunzy-ci} is closely
related to another result in \cite{Majadas16}. See \cref{rem:majadas}
for a more detailed comparison.

These results using the ghost condition notably parallel the results
obtained for contracting homomorphisms in
\cite{AvramovIyengarMiller06} and \cite{AvramovHochsterIyengarYao12},
see \cref{rem:contracting}.
\newpage
\begin{outline} The remainder of the paper is organized in the
  following way.
  \begin{itemize}
      \item In \cref{sec:cats}, we discuss the categorical tools that are
    employed to handle the homotopy theory of simplicial rings. This
    includes a discussion of functorial cylinder objects and simplicial
    model categories.
      \item In \cref{sec:aq}, we define \andq\ homology, and cover some basic
    properties and examples.
      \item In \cref{sec:frob-koszul}, we describe the relative
    Frobenius and Koszul complex in the simplicial setting.
      \item In \cref{sec:ghost-maps}, we define ghost maps, and cover
    some properties and examples.
      \item In \cref{sec:ghost-proof}, we prove
    \cref{thm:intro-ghost-ring}.
      \item In \cref{sec:singularities}, we review CI-dimension, prove
    \cref{thm:intro-kunzy} from which we deduce Kunz's theorem, and prove
    \cref{thm:intro-kunzy-ci}.
  \end{itemize}
\end{outline}

\begin{acknowledgements}
  This paper was produced in part as a PhD student at the University
  of Utah under the supervision of Srikanth Iyengar and Benjamin
  Briggs. I would like to thank both of them for their substantial
  guidance and review, without which this paper would not exist. I
  would also like to thank Peter McDonald for helping me detect a flaw
  in an earlier draft of this paper.
\end{acknowledgements}

%% file: cats.tex
\section{Categorical Homotopy Theory}\label{sec:cats}
In this section, we cover the categorical tools that will be used to
work with homotopies. We assume basic familiarity with model
categories and simplicial methods. Some useful references include
\cite{
  DwyerSpalinski95,
  GoerssJardine99,
  GoerssSchemmerhorn07,
  Riehl11,
  Riehl14,
  Quillen67
}.

\begin{notation}\label[notation]{not:cats} Let \(\ctC\) be a category.
  \begin{itemize}
      \item The category of simplicial objects in \(\ctC\) is denoted
    \(\cat{sC}\).
      \item The coproduct of \(X\) and \(Y\) in \(\ctC\) is denoted \(X \sqcup Y\).
      \item For any \(X \in \ctC\), let \(\ctC^X\) denote the under
    category (or coslice category) of \(X\), i.e. the category whose
    objects are maps \(X \xto{f} Y\) with domain \(X\) and whose maps
    are commutative triangles of the form
    \[
      \begin{tikzcd}[column sep=tiny]
        & X \ar[dl, "f"'] \ar[dr, "g"] & \\
        Y\ar[rr, "\alpha"'] && Z.
      \end{tikzcd}
    \]
    Dually, \(\ctC_X\) denotes the over (or slice) category of \(X\).
  \end{itemize}
\end{notation}

\begin{definition}\label[definition]{def:cyl-functor}
  Let \(\ctC\) be a model category. We say \(\Cyl \co \ctC \to \ctC\) is
  a cylinder functor if it satisfies the following conditions:
  \begin{itemize}
      \item \(\Cyl\) preserves colimits, cofibrations, and acyclic
    cofibrations (i.e. \(\Cyl\) is left-Quillen \cite[11.3]{Riehl14}).
      \item The fold map \(X \sqcup X \to X\) naturally
    factors through \(\Cyl X\)
    \[
      \begin{tikzcd}
        X \sqcup X \ar[r] & \Cyl X \ar[r, "\rho_X"] & X.
      \end{tikzcd}
    \]
      \item If \(X\) is cofibrant, then \(\rho_X \co \Cyl X \to X\) is
    a weak equivalence.
  \end{itemize}
  For a pair of maps \(f,g \co X \to Y\), we say \(f\) is
  \(\Cyl\)-homotopic to \(g\), written \(f \Cylsim g\), if there
  exists a map \(H \co \Cyl X \to Y\) making the following diagram
  commute.
  \[
    \begin{tikzcd}
      X \sqcup X \ar[r, "{(f,g)}"] \ar[d] & Y \\
      \Cyl X \ar[ur, dashed, "H"']
    \end{tikzcd}
  \]
\end{definition}
\begin{remark}\label[remark]{rem:cyl-ho}
  Let \(f,g \co X \to Y\) and suppose \(f \Cylsim g\).  For any map
  \(h \co X' \to X\) we obtain \(fh \Cylsim gh\) by restriction.
  \[
    \begin{tikzcd}
      X' \sqcup X' \ar[r] \ar[d] &  X \sqcup X \ar[r, "{(f,g)}"] \ar[d] & Y \\
      \Cyl X' \ar[r] & \Cyl X \ar[ur, dashed, "H"']
    \end{tikzcd}
  \]
  Furthermore, if we choose \(h \co X' \ontosim X\) to be a cofibrant
  replacement, we observe that \(f\) and \(g\) induce the same arrow
  in \(\cat{Ho}(\ctC)\).
\end{remark}

A category is pointed if it has an initial and terminal object and
these objects are isomorphic. An object with this property is called a
\emph{zero} object. If \(\ctC\) is pointed, then for any
\(X,Y \in \ctC\) the unique map factoring through the zero object is
called the \emph{trivial} map.
\begin{definition}\label[definition]{def:nullhomotopic}
  Let \(\ctC\) be a pointed model category and suppose \(f\) is a map
  in \(\ctC\). We say \(f\) is \emph{nullhomotopic} if it is trivial
  in \(\Ho(\ctC)\).
\end{definition}
\begin{remark}\label[remark]{rem:nullhomotopic}
  The definition given above does not require the existence of an
  explicit homotopy between \(f\) and the trivial map.
\end{remark}

\begin{definition}\label[definition]{def:simp-model-cat} A category \(\ctC\) is a
  \emph{simplicial category} if it is enriched, powered, and copowered
  in \(\sSet\)
  \begin{alignat*}{6}
    &\text{(enrichment)}\quad & \Map      \co & \ctC^\op  &&\times \ctC \to \sSet \\
    &\text{(power)}     \quad & [-,-]     \co & \sSet^\op &&\times \ctC \to \ctC \\
    &\text{(copower)}   \quad & - \odot - \co & \sSet     &&\times \ctC \to \ctC
  \end{alignat*}
  such that the triple \((\odot, [-,-], \Map)\) is a two variable adjunction,
  \[
    \ctC(K \odot X, Y)
    \cong \ctC(X, [K, Y])
    \cong \sSet(K,\Map(X,Y)).
  \]
  If \(\ctC\) is additionally a model category, then it is a
  \emph{simplicial model category} provided \((\odot, [-,-], \Map)\)
  is a two variable Quillen adjunction \cite[11.4.4]{Riehl14},
  \cite[II.3.1]{GoerssJardine99}.
\end{definition}

\begin{remark}\label[remark]{rem:sC-simp-model-cat}
  If \(\ctC\) is complete and cocomplete, then the category
  \(\cat{sC}\) of simplicial \(\ctC\)-objects is naturally endowed
  with the structure of a simplicial category where
  \[ (K \odot X)_n = \bigsqcup_{\sigma \in K_n} X_n \]
  and
  \[ \Map(X,Y)_n = \Hom_{\cat{sC}}(\Delta^n \odot X, Y). \]
  Let \(\eta \co [m] \to [n]\) be a map in the simplex category.
  Denote the corresponding simplicial structure map for \(K\) as
  \(\eta^K \co K_n \to K_m\) and likewise as \(\eta^X\) for \(X\). For
  \(\tau \in K_m\), let
  \[\mu_\tau \co \bigsqcup_{\sigma \in (\eta^K)^{-1}(\tau)} X_m \to X_m\]
  be the fold map. The corresponding simplicial structure map for the
  copower \(\eta^{K \odot X} \co (K \odot X)_n \to (K \odot X)_m\) is
  given by the composite
  \[
    \eta^{K \odot X} \co \bigsqcup_{\sigma \in K_n} X_n
    \xto{\bigsqcup_\sigma\! \eta^X} \bigsqcup_{\sigma \in K_n} X_m
    \xto{\cong} \bigsqcup_{\tau \in K_m} \left(\bigsqcup_{\sigma \in (\eta^K)^{-1}(\tau)} X_m\right)
    \xto{\bigsqcup_\tau\! \mu_\tau} \bigsqcup_{\tau \in K_m} X_m
  \]
  \hspace{-1em}\hspace{1em}\cite[II.2.5]{GoerssJardine99}.
\end{remark}

\begin{remark}\label[remark]{rem:cyl-simp}
  If \(\ctC\) is a simplicial model category, then
  \(\Cyl X \ceq \Delta^1 \odot X\) is a cylinder functor
  \cite[II.3.5]{GoerssJardine99}. Since \(\odot\) is a left Quillen
  bifunctor, if \(X\) is cofibrant, then \(- \odot X\) preserves
  acyclic cofibrations in \(\sSet\). The inclusions of the end points
  \(* \intosim \Delta^1\) are acyclic cofibrations, and hence the
  inclusions \(i_0, i_1 \co X \intosim \Cyl X\) are as
  well. Additionally, if \(X\) is cofibrant and \(Y\) is fibrant, then
  for any pair of maps \(f,g \co X \to Y\), if \(f = g\) in
  \(\Ho(\ctC)\), then \(f \Cylsim g\) \cite[II.3.6]{GoerssJardine99}.
\end{remark}

In all cases considered throughout this paper, there is a canonical
forgetful functor \(U \co \ctC \to \cat{Set}_*\) to pointed sets. The
extension \(U \co \cat{sC} \to \sSet_*\) allows one to define homotopy
groups in \(\cat{sC}\).

\begin{definition}\label[definition]{def:homotopy-groups}
  Let \(\ctC\) be a category with a forgetful functor
  \(U \co \ctC \to \cat{Set}_*\). For \(X \in \cat{sC}\), the
  \(i\)\tth\ homotopy group is defined to be
  \(\pi_i(X) \ceq \pi_i(UX)\). We say \(X\) is \emph{connected} if
  \(UX\) is connected, i.e. \(\pi_0(X) = *\).
\end{definition}

\begin{remark}\label[remark]{rem:pointed-homotopy-groups}
  If \(U \co \ctC \to \cat{Set}\) is a forgetful functor to the
  category of (non-pointed) sets with left adjoint
  \(F \co \cat{Set} \to \ctC\), and \(X\) is a simplicial
  \(\ctC\)-object, then one can define the \(i\)\tth\ homotopy group of
  \(X\) at the base point \(x \co F(*) \to X\) by
  \(\pi_i(X,x) \ceq \pi_i(UX, ux)\), where \(ux\) is the point in
  \(UX\) corresponding to the conjugate map \(x^\dag \co * \to UX\).
\end{remark}

\begin{lemma}\label[lemma]{lem:2cat-equiv}
  Let \(F \co \cat{C} \to \cat{D}\) be an equivalence of categories.
  If \(G,H \co \cat{D} \to \cat{C}\) are inverses to \(F\), then
  \(G \cong H\).
\end{lemma}
\begin{proof} Composing \(1_{\cat{D}} \cong FH\) and \(1_{\cat{C}} \cong GF\) as in
  the following diagram yields \(G \cong H\) as desired,
  \[
    \begin{tikzcd}[anchor=south]
      \cat{D}
      \ar[rr, bend left=45, "1_{\cat{D}}", ""'{name=B}]
      \ar[r, "H"]
      &
      \cat{C} \ar[to=B, marking, phantom, "\cong"] 
      \ar[r, "F"]
      \ar[rr, bend right=45, "1_{\cat{C}}"', ""{name=A}]
      &
      \cat{D} \ar[to=A, marking, phantom, "\cong"]
      \ar[r, "G"]
      &
      \cat{C}.
    \end{tikzcd}\qedhere
  \]
\end{proof}

\begin{lemma}\label[lemma]{lem:pullback-equiv}
  Suppose \(\ctC\) is a model category with cylinder functor
  \(\Cyl\). Fix a pair of maps \(f,g \co X \to Y\) with
  \(f \Cylsim g\). If \(X\) is cofibrant, then there is a natural
  isomorphism \(\LL f^* \cong \LL g^* \co \Ho(\ctC^Y) \to \Ho(\ctC^X)\). In
  particular, \(f \simeq g\) as objects in \(\ctC^X\).
\end{lemma}
\begin{proof}
  The cylinder diagram witnessing \(f \Cylsim g\) yields the following
  diagram of restriction functors.
  \[
    \begin{tikzcd}
      \ctC^X
      &
      \ctC^{\Cyl X}
      \ar[l, shift right, "i_0^*"']
      \ar[l, shift left, "i_1^*"]
      &
      \ctC^X
      \ar[l, "\rho_X^*"']
    \end{tikzcd}
  \]
  Since \(X\) and \(\Cyl X\) are cofibrant, and each of the maps
  \(i_0\), \(i_1\), and \(\rho\) are weak equivalences, all three
  restriction functors are Quillen equivalences, and so their left
  derived functors induce equivalences between their respective
  homotopy categories. Further, \(\rho i_0 = \rho i_1 = 1_X\), so by
  \cref{lem:2cat-equiv}, \(\LL i_0^* \cong \LL i_1^*\). Precomposing with
  \(\LL h^* \co \Ho(\ctC^Y) \to \Ho(\ctC^{\Cyl X})\) concludes the proof.
\end{proof}

\begin{remark}\label[remark]{rem:left-proper}
  If \(\ctC\) is left proper, the hypothesis in
  \cref{lem:pullback-equiv} that \(X\) is cofibrant can be
  omitted. Left-properness is equivalent to the condition that for any
  weak equivalence \(\widetilde{X} \tosim X\), the restriction
  \(\ctC^X \to \ctC^{\widetilde{X}}\) is a Quillen equivalence. Applying
  \cref{lem:pullback-equiv} to a cofibrant replacement
  \(\widetilde{X} \ontosim X\) thus yields the desired result.
\end{remark}

%% file: aq.tex
\section{Simplicial Rings and \andq\ Homology}\label{sec:aq}
In order to apply the techniques of homotopy theory to the study of
modules, we begin by embedding modules into a framework carrying a
suitable version of homotopy theory, i.e. chain complexes. Simplicial
rings provide an analogous framework for studying the homotopy theory
of rings. For more, see
\cite{
  Andre74,
  BasterraMandell05,
  GoerssJardine99,
  GoerssSchemmerhorn07,
  Iyengar07,
  SchwedeShipley00,
  Schwede01,
  QuillenHCR,
  Quillen67,
  Quillen70
}.

\begin{notation}\label[notation]{not:aq}
  In this section, in addition to the conventions of \cref{not:cats},
  we adhere to the following notation. Let \(\Ring\) denote the
  category of commutative unital rings. Fix \(k,\ell \in \Ring\). Any
  notation defined here for \(\Ring\) we define for \(\sRing\) as
  well.
  \begin{itemize}
      \item If \(k \to \ell\) is a map in \(\Ring\), then
    \(\Ring^k_\ell \ceq (\Ring^k)_{k \to \ell} \cong (\Ring_\ell)^{k \to \ell}\)
    has objects \(k \to R \to \ell\) (i.e. \(k\)-algebras with an
    augmentation to \(\ell\)) such that the composite is the specified
    map \(k \to \ell\).
      \item \(\Ring(k) \ceq \Ring^k_k\), the category of supplemented
    \(k\)-algebras (augmented \(k\)-algebras where the composite
    \(k \to k\) is the identity).
      \item If \(R \to S\) and \(R \to T\) are maps of rings, then
    \(S \otimes_R T\) is understood as the pushout in
    \(\Ring\). Likewise, \(S \lotimes_R T\) is the homotopy pushout in
    \(\sRing\).
      \item If \(R \to S\) is a map in \(\sRing_k\), then
    \(S \dslash R \ceq k \lotimes_R S\).
  \end{itemize}
\end{notation}

\begin{remark}\label[remark]{rem:pushouts-ring-mod}
  The pushout \(S \otimes_R T\) in \(\Ring\) corresponds to the usual
  tensor product of \(S\) and \(T\) as \(R\)-modules with the induced
  ring structure. Likewise, homotopy pushouts in \(\sRing\) are
  quasi-isomorphic to the corresponding derived tensor product of
  modules. In particular \(\pi_i(S \lotimes_R T) = \Tor^R_i(S,T)\).
  This can be seen from the fact that, if \(R \into S\) is a
  cofibration of simplicial rings, then \(S\)
  is cofibrant as a simplicial \(R\)-module.
\end{remark}

\begin{remark}\label[remark]{rem:hocofib}
  If \(f \co R \into S\) is a map in \(\sRing(k)\), then
  \(S \dslash R\) is the mapping cone, or homotopy cofiber, of \(f\)
  and fits into the homotopy cofiber sequence
  \[ R \xto{f} S \to S \dslash R. \]
  This construction depends on \(f\).
\end{remark}

Fix a simplicial ring \(k\).  Let \(\Ab(\sRing(k))\) denote the
category of abelian group objects in \(\sRing(k)\) and
\[\Ab \co \sRing(k) \to \Ab(\sRing(k))\]
be the abelianization functor. In \cite{QuillenHCR}, Quillen
establishes the equivalence
\[\Ab(\sRing(k)) \simeq \sMod(k).\]

Following the notion that homology is abelianized homotopy theory, we
arrive at the following definition, independently due to Andr\'{e}
\cite{Andre74} and Quillen \cite{QuillenHCR,Quillen70}.
\begin{definition}\label[definition]{def:andq}
  Fix a simplicial ring \(k\) and let \(R \in \sRing(k)\). Define
  \[\AQ^k(R) \ceq \LL \Ab(R).\]
  If \(M \in \sMod(k)\), then the \(i\)\tth\ \andq\ homology module of
  \(R\) with coefficients in \(M\) is
  \[ \AQ^k_i(R;M) \ceq \pi_i\left(M \lotimes_k \AQ^k(R)\right), \]
  (see \cref{def:homotopy-groups}). When the base ring \(k\) is clear
  from context, we write \(\AQ\) in place of \(\AQ^k\). Likewise, if
  \(M = k\), we write \(\AQ_i(R)\) in place of \(\AQ_i(R;M)\).
\end{definition}
\begin{remark}\label[remark]{rem:AQ-complex}
  Under the identification \(\Ab(\sRing(k)) \simeq \sMod(k)\), it can
  be shown that \(\Ab(A) = I/I^2\) where \(I = \ker(A \to k)\).
  Additionally, the Dold-Kan theorem
  \cite[III.2.5]{GoerssJardine99}
  \cite[Sec 2.6]{Quillen67}
  \cite[4.1]{GoerssSchemmerhorn07} provides
  us with the Quillen equivalence
  \[ N \co \sMod(k) \tosim \Ch_{\ge 0}(Nk). \]
  Thus, if \(A \ontosim R\) is a cofibrant replacement and
  \(I = \ker(A \to k)\), we may compute \andq\ homology as the
  homology of a chain complex:
  \[ \AQ_i(R) = \h_i(N(I/I^2)). \]
\end{remark}
\begin{remark}\label[remark]{rem:AQ-cotangent}
  Given a map \(R \to S\) of rings, considered as the terminal map in
  \(\sRing_S\), note that \(S \dslash R = S \lotimes_R S\). We recover
  the cotangent complex as
  \[L_{S|R} = \AQ^S(S \dslash R).\]
  Furthermore, for \(R \in \sRing(k)\),
  \[ \AQ^k(R) = k \lotimes_R L_{R|k}. \]
\end{remark}
\begin{remark}\label[remark]{rem:AQ-local}
  When \(k\) is a field, our definition of \(\AQ\) in \cref{def:andq}
  should be regarded as a fiber-wise definition. Given a map of schemes
  \(f \co X \to Y\), one can construct a complex \(\mathcal{L}_{X|Y}\)
  of sheaves of \(\mathcal{O}_X\) modules whose fiber at
  \(x \co \Spec(k) \to X\) is quasi-isomorphic to
  \(\AQ^k(\mathcal{O}_{X,x} \dslash \mathcal{O}_{Y,f(x)})\). Furthermore,
  if \(X = \Spec(S)\) and \(Y = \Spec(R)\) as in
  \cref{rem:AQ-cotangent}, then \(L_{S|R}\) is quasi-isomorphic to the
  global sections of \(\mathcal{L}_{X|Y}\)
  \cite[\href{https://stacks.math.columbia.edu/tag/08T3}{Tag
    08T3}]{StacksProject}.
\end{remark}
\begin{example}\label[example]{ex:AQ-kx/x2}
  Let \(k\) be a ring and \(R = k[x]/(x^2)\). Let
  \(A = k[x, y \mid \partial y = x^2]\) be the simplicial polynomial
  ring generated by \(x\) in degree \(0\) and \(y\) in degree
  \(1\). Direct computation verifies that \(A \ontosim R\) is a
  cofibrant replacement. By \cref{rem:AQ-complex}, \(\Ab(A)\) is given
  by
  \[ I/I^2 \cong kx \oplus ky, \]
  so
  \[ N(I/I^2) = 0 \to k \xto{0} k \to 0, \]
  and thus
  \[
    \AQ_i(R) =
    \begin{cases}
      k & i = 0,1 \\
      0 & i \ge 2.
    \end{cases}
  \]
\end{example}

\begin{construction}\label[construction]{cons:jacobi-zariski}
  \andq\ homology satisfies the Eilenberg-Steenrod axioms for a
  generalized reduced homology theory. In particular, homotopy cofiber
  sequences yield long exact sequences in \andq\ homology. Fix a
  simplicial ring \(k\) and let
  \[R \to S \to T\]
  be a sequence of maps in \(\sRing_k\). Then
  \[ S \dslash R \to T \dslash R \to T \dslash S \]
  is a homotopy cofiber sequence in \(\sRing(k)\). If \(M\) is a
  simplicial \(k\)-module, then the induced long exact sequence
  \[
    \cdots
    \to \AQ^k_i(S \dslash R; M)
    \to \AQ^k_i(T \dslash R; M)
    \to \AQ^k_i(T \dslash S; M)
    \to \AQ^k_{i-1}(S \dslash R; M)
    \to \cdots
  \]
  is known as the \emph{Jacobi-Zariski} sequence (cf. \cite{Iyengar07}).
\end{construction}

\begin{remark}\label[remark]{rem:completion}
  Let \(S\) be a noetherian local ring with
  residue field \(k\) and let \(\widehat{S}\) be the completion of \(S\)
  at the maximal ideal. Then
  \[\widehat{S} \dslash S = k \lotimes_S \widehat{S} \simeq k \otimes_S \widehat{S} = k,\]
  so \(\AQ^k(\widehat{S} \dslash S) \simeq \AQ^k(k) \simeq 0\). If
  \(R \to S\) is any map of rings, then by applying the Jacobi-Zariski
  sequence from \cref{cons:jacobi-zariski} to \(R \to S \to \widehat{S}\)
  we conclude
  \(\AQ^k(S \dslash R) \simeq \AQ^k(\widehat{S} \dslash R)\).
\end{remark}

%% file: frob_koszul.tex
\section{The Relative Frobenius and the Koszul Complex}\label{sec:frob-koszul}
In this section we review the relative Frobenius and the Koszul
complex in the simplicial setting, and highlight their relationship.

\begin{notation}\label[notation]{not:frob-koszul}
  We adopt the notation of \cref{not:cats} and \cref{not:aq} for
  this section. Additionally, if \(A \xto{f} B\) is a ring map, then
  \(B_f\) denotes an identical copy of \(B\) with an \(A\)-algebra
  structure given via \(f\). Note that, since the underlying ring
  is the same, \(\Ring(B_f) = \Ring(B)\).
\end{notation}

\begin{construction}[Relative Frobenius]\label{con:rel-frob}
  Suppose \(k \to R\) is a map of simplicial rings of prime
  characteristic \(p\), and let \(F\) denote the Frobenius
  endomorphism. Commutativity of the diagram
  \[
    \begin{tikzcd}
      k \ar[r, "F"] \ar[d] & k_F \ar[d] \\
      R \ar[r, "F"] & R_F
    \end{tikzcd}
  \]
  yields the map
  \[F_{R|k} \co k_F \lotimes_k R \to R_F\]
  in \(\Ho(\sRing^{k_F})\), known as the %
  \emph{derived relative Frobenius} (cf. \citestacks{0CC9}).
\end{construction}

\begin{construction}\label[construction]{con:Koszul}
  Let \(R\) be a ring and \(\bff = f_1, \ldots, f_n\) be a sequence of
  elements in \(R\). Using the notation of \cref{ex:AQ-kx/x2}, define
  the \emph{simplicial Koszul complex}, \(K^R(\bff)\), of \(\bff\) to
  be the simplicial polynomial ring
  \[K^R(\bff) \ceq R[\xi_1, \ldots, \xi_n \mid \partial\xi_i = f_i].\]
  If \(I \le R\) is an ideal and \((\bff) = I\) is a minimal
  generating set, we define \(K^R(I) = K^R(\bff)\) and note that by the
  standard change of basis isomorphism, this construction is
  independent of the choice of minimal generators.
\end{construction}

\begin{remark}\label[remark]{rem:Koszul-is-Koszul}
  Treating \(K^R(\bff)\) as a simplicial \(R\)-module and applying the
  Dold-Kan theorem (see \cref{rem:AQ-complex}), recovers the usual
  Koszul complex of \(R\). This can be seen by direct verification
  when \(n = 1\) \cite[4.16]{Iyengar07}. Since
  \[K^R(\bff) \cong K^R(f_1) \otimes_R \cdots \otimes_R K^R(f_n), \]
  the general claim follows by induction. In particular, we note that
  \[
    \begin{cases}
      \pi_0(K^R(\bff)) \cong R/(\bff) \\
      \pi_i(K^R(\bff)) = 0 & i > n.
    \end{cases}
  \]
\end{remark}

\begin{remark}\label[remark]{rem:Koszul-complete}
  Each \(K^R(f_i)\) is a model for the homotopy cofiber in \(\sRing(R)\)
  \[ R[y_i] \xto{y_i \mapsto f_i} R \to K^R(f_i) \]
  and hence, as a graded \(R\)-module,
  \(f_i \cdot \pi_*(K^R(f_i)) = 0\). Iterating this process over the
  sequence \(\bff\) (in any order) recovers \(K^R(I)\), and hence
  \[ I\pi_*(K^R(I)) = 0. \] %
  Therefore \(\pi_*K^R(I)\) is \(I\)-adically complete as a graded
  \(R\)-module. If \(R\) is noetherian then \(\widehat{R}_I\) is flat over
  \(R\) and by \cref{rem:Koszul-is-Koszul}, \(\pi_*(K^R(I))\) is
  finite over \(R\), so
  \[
    \pi_*(K^R(I))
    \cong \widehat{R}_I \otimes_R \pi_*(K^R(I))
    \cong \pi_*(\widehat{R}_I \otimes_R K^R(I))
    \cong \pi_*(K^{\widehat{R}}(\widehat{I})).
  \]
  Thus, \(K^R(I) \simeq K^{\widehat{R}}(\widehat{I})\) as \(R\)-algebras
  (see for example \cite[3.17]{BalmerSanders24}).
\end{remark}

\begin{construction}\label[construction]{con:Koszul-cohen}
  Let \((R, \mfm, k)\) be a noetherian local ring. In this setting, we
  write \(K^R \ceq K^R(\mfm)\).

  By the Cohen structure theorem, \(\widehat{R} \cong Q/I\) where
  \(Q = \Lambda \ldsq x_1, \ldots, x_d \rdsq\), and \(\Lambda\) is
  either \(k\) or a discrete valuation ring with uniformizer \(p\) a
  prime number \citestacks{032A}. Additionally, \(Q\) can be chosen
  minimally, in the sense that \(\{x_1, \ldots, x_d, p\}\) is a
  minimal generating set for \(\mfm_{\widehat{R}}\). Thus if \(Q\) is
  minimal, we may directly verify
  \[k \lotimes_Q \widehat{R} \simeq K^Q \otimes_Q \widehat{R} \simeq K^{\widehat{R}} \simeq K^R,\]
  where the final equivalence follows from \cref{rem:Koszul-complete}.
  
  If \(\phi \co (R,k) \to (S,\ell)\) is a map of noetherian local
  rings, define the map
  \[ K^\phi \co \ell \lotimes_k K^R \to K^S, \]
  in \(\Ho(\sRing(\ell))\), where the map \(k \to K^R\) is given by
  \(k \to k \lotimes_Q \widehat{R} \simeq K^{\widehat{R}} \simeq
  K^R\).

  Working out this construction explicitly involves choosing lifts to
  the regular rings coming from the Cohen structure theorem and other
  cofibrant replacements, but after tensoring with the residue fields,
  the result is independent of these choices in \(\Ho(\sRing)\).
\end{construction}

\begin{remark}\label[remark]{rem:frob-is-koszul}
  Let \((R,k)\) be a local ring with prime characteristic
  \(p\). Let \(F \co R \to R\) denote the Frobenius endomorphism and
  fix a minimal Cohen presentation \(Q \to \widehat{R}\). The induced map
  \(K^F\) on the Koszul complex is given by
  \(k_F \lotimes_{Q_F} F_{\widehat{R}|Q}\).
  \[
    \begin{tikzcd}[displaystyle]
      k_F \lotimes_k K^R \ar[d,"K^F"'] \ar[r, phantom, sloped, "\simeq"]
      & k_F \lotimes_k K^{\widehat{R}}  \ar[d] \ar[r, phantom, sloped, "\simeq"]
      & k_F \lotimes_k k \lotimes_Q \widehat{R} \ar[r, phantom, sloped, "\simeq"]
      & k_F \lotimes_{Q_F} Q_F \lotimes_Q \widehat{R} \ar[d, "k_F \lotimes_{Q_F} F_{\widehat{R}|Q}"]
      \\
      K^{R_F} \ar[r, phantom, sloped, "\simeq"]
      &K^{\widehat{R}_F} \ar[rr, phantom, sloped, "\simeq"]
      && \mathclap{k_F \lotimes_{Q_F} \widehat{R}_F}
    \end{tikzcd}
  \]
\end{remark}

%% file: ghost_maps.tex
\section{Ghost Maps}\label{sec:ghost-maps}
As discussed in \cref{sec:aq}, the natural theory of homology in the
setting of commutative rings is \andq\ homology. By analogy to the
definition for ghost maps in the setting of chain complexes \cite{
  Christensen98,
  Beligiannis08,
  Kelly65,
  Letz21,
  LiuPollitz23,
  OppermannStovicek12
}, we define a ring map to be ghost by the vanishing of \andq\ homology.

\begin{notation}\label[notation]{not:ghost-maps}
  We adopt the notation of \cref{not:cats}, \cref{not:aq}, and
  \cref{not:frob-koszul} for this section.
\end{notation}

\begin{definition}\label[definition]{def:ghost}
  A map \(R \to S\) in \(\sRing(k)\) is \emph{ghost} if the induced map
  on \andq\ homology \(\AQ_*(R) \to \AQ_*(S)\) is zero.
\end{definition}

\begin{example}\label[example]{exm:etale-ghost}
  If \(R \to S\) is \'{e}tale, then \(\AQ_*(S \dslash R) = 0\), so any map to
  or from \(S \dslash R\) is ghost.
  \citestacks{08R2}
\end{example}

\begin{proposition}\label[proposition]{pro:frob-ghost}
  Let \(k\) and \(\ell\) be simplicial rings of prime characteristic
  and suppose \(k \to R \to \ell\) is an augmented simplicial
  \(k\)-algebra. Then the map
  \[
    \ell_F \lotimes_{k_F} F_{R|k} \co \ell_F \lotimes_k R \to \ell_F \lotimes_{k_F} R_F
  \]
  in \(\Ho(\sRing(\ell_F))\) is ghost.
\end{proposition}
\begin{proof}
  Choose a cofibrant replacement \(k[X] \ontosim R\) over \(k\), and
  note that the same underlying simplicial ring yields a cofibrant
  replacement \(k_F[X] \ontosim R_F\) over \(k_F\).  We thus compute
  \[
    \ell_F \otimes_k k[X] = \ell_F[X]
    \mand
    \ell_F \otimes_{k_F} k_F[X] = \ell_F[X].
  \]
  The Frobenius on \(k[X]\) naturally lifts the Frobenius on \(R\)
  \[
    \begin{tikzcd}
      k[X] \ar[r, "F"]  \ar[d, two heads, sloped, "\sim"]
      & k_F[X] \ar[d, two heads, sloped, "\sim"] \\
      R \ar[r, "F"'] & R_F,
    \end{tikzcd}
  \]
  so \(\ell_F \lotimes_{k_F} F_{R|k}\) is given by
  \[
    \begin{tikzcd}[row sep = 0.3ex]
      \ell_F[X] \ar[r] & \ell_F[X] \\
      lx \ar[r, maps to] & lx^p.
    \end{tikzcd}
  \]
  If \(I\) and \(I_F\) are their respective augmentation ideals, then
  the induced map \(I \to I_F\) factors through \(I_F^p\), so the map
  \(I/I^2 \to I_F/I_F^2\) is zero. Thus
  \(\ell_F \lotimes_{k_F} F_{R|k}\) is ghost by \cref{rem:AQ-complex}.
\end{proof}

We conclude this section with an illustration of the rigidity
experienced by ghost maps in the case of ordinary rings, and,
following \cref{rem:kunzish-ghost}, contrast this with the weaker
hypothesis that the induced map on the Koszul complex is ghost.

\begin{proposition}\label[proposition]{pro:vanish-conormal}
  Fix a ring \(R\) and a map \(\phi \co S \to T\) in \(\Ring(R)\). Let
  \[I_S = \ker(S \to R) \mand I_T = \ker(T \to R).\]
  If
  \(\phi\) is ghost then \(\phi(I_S) \subset I_T^2\).
\end{proposition}
\begin{proof}
  Since \(S \to R\) is surjective, \cite[3.14]{QuillenHCR} computes
  \(\AQ_1(R \dslash S) \cong I_S/I_S^2 \). The Jacobi-Zariski
  sequence in \cref{cons:jacobi-zariski} corresponding to the
  suspension homotopy cofiber sequence
  \(S \to R \to R \dslash S\)
  gives us \(\AQ_0(S) \cong \AQ_1(R \dslash S)\), and so we compute
  \[\AQ_0(S) \cong I_S/I_S^2.\]
  Likewise, we obtain \(\AQ_0(T) \cong I_T/I_T^2\).  Since \(\phi\) is
  ghost, the induced map
  \[I_S/I_S^2 \to I_T/I_T^2\]
  is zero, so \(\phi(I_S) \subset I_T^2\) as desired.
\end{proof}

\begin{definition}\label[definition]{def:contracting}
  If \((R, \mfm, k)\) is a local ring, a local endomorphism
  \(\phi \co R \to R\) is contracting if \(\phi^j(\mfm) \subset \mfm^2\)
  for \(j \gg 0\) \cite{AvramovIyengarMiller06}.
\end{definition}

\begin{corollary}\label[corollary]{cor:contracting}
  Suppose \(R\) is an equicharacteristic local ring with residue field
  \(k\) and let \(\phi\) be a local endomorphism. If \(\phi\) is
  ghost as a map in \(\Ring(k)\), then \(\phi\) is contracting.
\end{corollary}

We now pivot our discussion to the case where \(K^\phi\) is ghost.

\begin{remark}\label[remark]{rem:kunzish-ghost}
  Suppose \(R\) is a complete local ring and let \(Q \to R\) be a
  minimal Cohen presentation as in \cref{con:Koszul-cohen}, where
  \(Q = \Lambda \ldsq x_1, \ldots, x_d \rdsq\). Since \(Q\) is
  regular, \(\AQ(Q \dslash \Lambda) \simeq (\bfx)/(\bfx)^2\). In
  particular, \(\AQ_{\ge 1}(Q \dslash \Lambda) = 0\). Further,
  minimality of \(Q \to R\) implies
  \(\AQ_0(Q \dslash \Lambda) \cong \AQ_0(R \dslash \Lambda)\), and so
  the long exact sequence in \andq\ homology corresponding to the
  homotopy cofiber sequence
  \[Q \dslash \Lambda \to R \dslash \Lambda \to R \dslash Q\]
  shows that \(\AQ_i(R \dslash \Lambda) \cong AQ_i(R \dslash Q)\) for
  \(i \ge 1\). Thus, since \(R \dslash Q \simeq K^R\), the condition
  that \(K^\phi\) is ghost is equivalent to
  \(\AQ_{\ge 1}(\phi \dslash \Lambda) = 0\). In particular, this
  condition holds when \(\phi \dslash \Lambda\) is itself ghost.
\end{remark}

\begin{remark}\label[remark]{rem:conormal-ghost}
  Suppose \(R \in \Ring(k)\) is a local complete intersection ring with
  residue field \(k\). By definition, \(\hat{R}\) admits a minimal
  Cohen presentation \(\hat{R} \cong Q/I\) where \(I\) is generated by a
  regular sequence. A direct calculation similar to that in the proof
  of \cref{pro:vanish-conormal} shows that
  \[
    \AQ_i(R) \cong
    \begin{cases}
      \mfm_Q/\mfm_Q^2 & i=0 \\
      I/I^2 & i=1 \\
      0 & i \ge 2
    \end{cases}
  \]
  Thus by \cref{rem:kunzish-ghost}, if \(\phi \co R \to R\) is a local
  endomorphism, then \(K^\phi\) is ghost if and only if
  \(\tilde{\phi}(I) \subset I^2\), where \(\tilde{\phi}\) is a lift of
  \(\phi\) to \(Q\). If \(R\) is mixed characteristic with
  \(\Lambda \subset Q\) as in \cref{con:Koszul-cohen}, then a similar
  statement exists for \(\AQ(R \dslash \Lambda)\) and one deduces that
  \(K^\phi\) is ghost if and only if \(\tilde{\phi}(I) \subset pI + I^2\).
\end{remark}

\begin{example} The following are examples of local complete
  intersection rings \(R \in \Ring(k)\) with residue field \(k\) and
  local endomorphisms \(\phi\). In each example, the map \(\phi\) is
  not ghost, but the induced map on the Koszul complex \(K^\phi\) is a
  ghost map (see \cref{cor:contracting} and
  \cref{rem:conormal-ghost}).
  \begin{itemize}
      \item \(R = k \ldsq x,y \rdsq/(y^3)\) and \(\phi =
    \begin{cases}
      x \mapsto x \\
      y \mapsto y^2
    \end{cases}
    \)
      \item \(R = k \ldsq x,y \rdsq /(xy)\) and \(\phi =
    \begin{cases}
      x \mapsto x \\
      y \mapsto 0
    \end{cases}
    \)
      \item \(R = k \ldsq x,y,z \rdsq /(xyz)\) and \(\phi =
    \begin{cases}
      x \mapsto x \\
      y \mapsto y^2 \\
      z \mapsto xz^2
    \end{cases}
    \)
  \end{itemize}
\end{example}

\begin{example}\label[example]{ex:non-contracting-ghost}
  Fix \(p\) a prime and let \(R = \bbZ_p[x, y]/(xy)\). Consider the
  endomorphism \(\phi\) defined by \(x \mapsto x\) and
  \(y \mapsto py\). By \cref{rem:conormal-ghost}, \(K^\phi\) is ghost.
\end{example}

%% file: ghost_lemma.tex
\section{A Ghost Lemma for Commutative Rings}\label{sec:ghost-proof}
We now state the main result of this paper, a \textit{ghost lemma} for
commutative rings, analogous to \cite[Theorem 3]{Kelly65} (see also
\cite[Theorem 8.3]{Christensen98}). Recall that a nullhomotopic map is
one which is trivial in the homotopy category
(\cref{def:nullhomotopic}) and a trivial map in \(\sRing(k)\) is one
which factors through \(k\).

\begin{theorem}\label[theorem]{thm:ghost-ring}
  Let \(k\) be a field. Suppose \(S \in \sRing(k)\) is connected and
  \(\pi_i(S) = 0\) for all \(i \ge 2^n\) for some \(n\).  If
  \(f \co R \to S\) is a composition of at least \(n\) ghost maps,
  then \(f\) is nullhomotopic.
\end{theorem}

The remainder of this section is devoted to the proof of the above
theorem.
\begin{notation}\label[notation]{not:ghost-proof}
  In addition to the notation of \cref{not:cats} and \cref{not:aq},
  we adhere to the following notation.
  \begin{itemize}
      \item \(\NRing(k)\) is the category of nonunital commutative
    \(k\)-algebras.
      \item \(\NRR \co \NRing(k) \to \Ring(k)\) is the functor
    \(I \mapsto k \oplus I\), where the multiplicative structure is
    given by \((a+i)(b+j) = ab + (aj + bi + ij)\).
      \item \(\Cyl_\Ring\), \(\Cyl_\NRing\), and \(\Cyl_\Mod\)
    denote the simplicial cylinder functors for \(\sRing(k)\),
    \(\sNRing(k)\), and \(\sMod(k)\) respectively (\cref{rem:cyl-simp}).
      \item For each of the above cylinder functors, we adopt the
    shorthand \(f \Cylsim[\ctC] g\) for \(f \Cylsim[\Cyl_\ctC] g\).
  \end{itemize}
  For every \(R \in \Ring(k)\), there is a splitting of \(k\)-modules
  \(R = k \oplus I\), where \(I\) is the augmentation ideal
  \(\ker(R \to k)\). Consequently, the functor
  \(\NRR \co \NRing(k) \to \Ring(k)\) is an equivalence of categories.
\end{notation}

\begin{remark}\label[remark]{rem:homotopy-connected}
  Fix a ring map \(k \to \ell\). The free functor
  \(F \co \cat{Set} \to \Ring^k_\ell\) takes a set \(X\) to the
  polynomial ring \(k[X]\) equipped with the inclusion \(k \to k[X]\)
  and the augmentation \(k[X] \to \ell\) sending the elements of \(X\)
  to 0. The right adjoint of this functor is the forgetful functor
  \(U \co \Ring^k_\ell \to \cat{Set}\). Contrary to what one might
  expect, we can conclude from the above description of \(F\) that
  \(U(R)\) is the underlying set of the augmentation ideal
  \(I = \ker(R \to \ell)\) rather than the underlying set of \(R\)
  itself. The pointed version of the forgetful functor is defined
  identically, where 0 is identified as the marked point. The homotopy
  groups of \(R\) as defined in \cref{def:homotopy-groups} are thus
  the homotopy groups of \(I\) as a simplicial set. Consequently, for
  \(R\) to be connected as an object in \(\sRing^k_\ell\), we ask that
  \(\pi_0(I) = 0\).
\end{remark}

Let \(R \to S \) be a map in \(\sRing(k)\) and let \(I\) and \(J\) be
their respective augmentation ideals. We sketch an
outline below for the proof of the theorem in two stages.
\begin{center}
  \renewcommand{\arraystretch}{1.5}
  \begin{tabular}{r p{0.8\linewidth}}
    \textbf{Stage 1}  & \(R \to S\) is ghost \\
    \(\implies\) & \(I/I^2 \to J/J^2\) is nullhomotopic as a module homomorphism \\
    \(\implies\) & \(I/I^2 \to J/J^2\) is nullhomotopic as an algebra homomorphism (\ref{lem:mod-to-alg}) \\
    \(\implies\) & \(I \to J\) factors (up to homotopy) through \(J^2\) as an algebra homomorphism (\ref{lem:I-lift-J2})
  \end{tabular}
\end{center}
\begin{center}
  \renewcommand{\arraystretch}{1.5}
  \begin{tabular}{r p{0.8\linewidth}}
    \textbf{Stage 2} & \(R = R_0 \to R_1 \to \cdots \to R_{n} = S\) is a sequence of ghost maps \\
    \(\implies\) & \(I \to J\) factors (up to homotopy) through \(I_0 \to I_1^2 \to \cdots \to I_{n}^{2^n}\) \\
    \(\implies\) & \(R \to S\) factors (up to homotopy) through \(\NRR(J^{2^n})\) \\
    \(\implies\) & \(R \to S\) is nullhomotopic (\ref{lem:homotopy-powers}, \ref{lem:null-inc}).
  \end{tabular}
\end{center}

\begin{remark}\label[remark]{rem:compare-to-bilmp}
  \cref{thm:ghost-ring} can be considered a strengthening of
  \cite[Prop 3.7]{BILMP2023} when applied to the Frobenius
  endomorphism. The Frobenius endomorphism directly yields an algebra
  homomorphism \(I \to J^2\) by restriction. However, the ghost
  condition considered in this section \textit{a priori} only yields
  this lift up to homotopy, and critically, a module homotopy rather
  than an algebra homotopy. This distinction is the primary source of
  additional technical complexity in this section.
\end{remark}

\begin{lemma}\label[lemma]{lem:homotopy-powers}
  Let \(I \in \sNRing(k)\) be connected and cofibrant. Then
  \(\pi_i(I^n) = 0\) for \(i < n\).
\end{lemma}
\begin{proof}
  Let \(R = \NRR(I)\). Since \(R\) is cofibrant, each \(R_i\) is a
  retract of a (possibly infinite) polynomial \(k_i\)-algebra. By
  \cite[Theorem 6.12]{Quillen70}, it is sufficient to show that \(I\)
  is degree-wise regular in \(R\) (explained below), so we reduce to
  the non-simplicial setting.
 
  Assume \(k\) is an ordinary ring and let \(R \in \Ring(k)\). Suppose
  for some (possibly infinite) set \(X\), \(R\) admits a retract
  \(R \to k[X] \to R\) in \(\Ring(k)\). In this setting, by
  \cite[Theorem 6.14]{Quillen70}, the regularity condition above asks
  that \(I/I^2\) is projective over \(k\) and
  \(\AQ(k \dslash R) \simeq \Sigma I/I^2\).
  
  Since \(R \to k[X] \to R\) is a retract, \(I/I^2\) is a direct
  summand of \((X)/(X)^2\). However, \((X)/(X)^2 \cong k^X\) is free
  over \(k\), and hence \(I/I^2\) is projective.
  
  Applying \(\AQ\) to the suspension homotopy cofiber
  sequence \(R \to k \to k \dslash R \) shows
  \[ \AQ(k \dslash R) \simeq \Sigma \AQ(R), \]
  and since \(R\) is cofibrant over \(k\), \(\AQ(R) \simeq I/I^2\).
\end{proof}

\begin{lemma}\label[lemma]{lem:null-inc}\cite[Sec. 3]{BILMP2023}
  Let \(J\) be in \(\sNRing(k)\), and \(I\) a subalgebra of \(J\).
  Suppose \(\pi_i(I) = 0\) for \(i < n\) and \(\pi_i(J) = 0\) for
  \(i \ge n\). Then the inclusion \(I \to J\) is nullhomotopic as a
  map of algebras.
\end{lemma}

\begin{proof}
  Let \(J/I \into K\) be the free extension obtained by killing cycles
  in degrees \(i > n \) \cite[4.10]{Iyengar07}.  The homotopy fiber
  sequence \(I \to J \onto J/I\) produces a long exact sequence of
  homotopy groups
  \[
    \cdots \to \pi_i(I) \to \pi_i(J) \to \pi_i(J/I) \to \pi_{i-1}(I) \to \cdots.
  \]
  For \(i < n\), there are isomorphisms \(\pi_i(J/I) = \pi_i(J)\) due
  to the hypothesis \(\pi_i(I) = \pi_{i-1}(I) = 0\). Additionally,
  when \(i = n\), we see \(\pi_i(J) = 0\) and \(\pi_{i-1}(I) = 0\), so
  \(\pi_i(J/I) = 0\).  We conclude that
  \(\pi_i(J) \cong \pi_i(J/I) \cong \pi_i(K)\) for all \(i \le
  n\). Since \(\pi_i(J) = 0 = \pi_i(K)\) for \(i > n\), the composite
  \(J \to K\) is a quasi-isomorphism. The proof follows from
  commutativity of the following diagram.
  \[
    \begin{tikzcd}[anchor=south, column sep={4.5em, between origins}]
      I \ar[r] \ar[rr, bend left, anchor=center, "0" ] 
      & J \ar[r] \ar[rr, bend right, sloped, anchor=center, "\sim"']
      & J/I \ar[r]
      & K
    \end{tikzcd}\qedhere
  \]  
\end{proof}

\begin{lemma}\label[lemma]{lem:mod-to-alg}
  Let \(\phi \co I \to J\) be a map in \(\sNRing(k)\) and let
  \(\bar{\phi}\) denote the induced map on the quotients
  \(I/I^2 \to J/J^2\). Then \(\bar{\phi} \Cylsim[\Mod] 0\)
  implies \(\bar{\phi} \Cylsim[\NRing] 0.\)
\end{lemma}

\begin{proof}
  Assume \(\bar{\phi} \Cylsim[\Mod] 0\). Since
  \(\NRR \co \NRing(k) \to \Ring(k)\) is an equivalence of categories,
  it is sufficient to show \(\NRR\bar{\phi} \Cylsim[\Ring] 0.\)
  
  Let \(h \co \Cyl_\Mod (I/I^2) \to J/J^2\) be a
  \(\Cyl_\Mod\)-homotopy witnessing \(\bar{\phi} \Cylsim[\Mod]
  0\). Under the identification \(\Delta^1_n = \{0,\ldots,n+1\}\), we
  use \cref{rem:sC-simp-model-cat} and \cref{rem:cyl-simp}, to express
  \(h\) in simplicial degree \(n\) as
  \[ h_n \co \bigoplus_{i=0}^{n+1} I_n/I_n^2 \to J_n/J_n^2. \] %
  Let \(h_{n,i}\) denote the \(i\)-th component map, and let
  \(\sigma\) denote the sum map. We write \(h_n\) as the composite
  \[
    \bigoplus_{i=0}^{n+1} I_n/I_n^2
    \xto{\bigoplus_i h_{n,i}}
    \bigoplus_{i=0}^{n+1} J_n/J_n^2
    \xto{\sigma}
    J_n/J_n^2
  \]
  Since both
  \(I/I^2\) and \(J/J^2\) have trivial multiplication, each component
  map is a non-unital algebra homomorphism. Define
  \[
    H_n \co \bigotimes_{i=0}^{n+1} \NRR(I_n/I_n^2) \to \NRR(J_n/J_n^2)
  \]
  component-wise by \(H_{n,i} = \NRR h_{n,i}\).

  We claim that \(H \co \Cyl_\Ring(\NRR(I/I^2)) \to \NRR(J/J^2)\) is
  compatible with the simplicial structure maps, that is, for
  \(\eta \co [m] \to [n]\) in the simplex category we must show
  \[
    H_m \circop \eta^{\Cyl_\Ring(\NRR(I/I^2))}
    = \eta^{\NRR(J/J^2)}\circop H_n.
  \]
  Using the formula in \cref{rem:sC-simp-model-cat}, and the fact that
  the fold map in \(\Ring(k)\) is given by multiplication, we compute
  \(H_m \circop \eta^{\Cyl_\Ring(\NRR(I/I^2))}\) and
  \(\eta^{\NRR(J/J^2)} \circop H_n\) as the diagonal composites in the
  diagrams \cref{dia:htpy} and \cref{dia:htpy2} respectively.
  \begin{gather}
    \label{dia:htpy}\tag{TR}
    \begin{tikzcd}[column sep=large, ampersand replacement=\&]
      \bigotimes_{i=0}^{n+1} \NRR(I_n/I_n^2)
      \ar[r, "\bigotimes_i\! \eta^{\NRR(I\sslash{/}I^2)}"]
      \ar[dr, "\bigotimes_i\! H_{m,\eta^{\Delta^1}\!(i)} \circop \eta^{\NRR(I\sslash{/}I^2)}"']
      \& \bigotimes_{i=0}^{n+1} \NRR(I_m/I_m^2)
      \ar[r, "\bigotimes_j\! \mu_j"]
      \ar[d, "\bigotimes_i\! H_{m,\eta^{\Delta^1}\!(i)}"]
      \& \bigotimes_{j=0}^{m+1} \NRR(I_m/I_m^2)
      \ar[d,"\bigotimes_j\! H_{m,j}"]
      \\
      \& \bigotimes_{i=0}^{n+1} \NRR(J_m/J_m^2)
      \ar[r, "\bigotimes_j\! \mu_j"]
      \ar[dr, "\mu"']
      \&
      \bigotimes_{j=0}^{m+1} \NRR(J_m/J_m^2)
      \ar[d, "\mu"]
      \\
      \&\& \NRR(J_m/J_m^2)
    \end{tikzcd}
    \\[-5ex]
    \label{dia:htpy2}\tag{BL}
    \begin{tikzcd}[column sep=huge, ampersand replacement=\&]
      \bigotimes_{i=0}^{n+1} \NRR(I_n/I_n^2)
      \ar[d,"\bigotimes_i\! H_{n,i}"']
      \ar[dr, "\bigotimes_i\! \eta^{\NRR(J\sslash{/}J^2)} \circop H_{n,i}"]
      \& \& \\
      \bigotimes_{i=0}^{n+1} \NRR(J_n/J_n^2)
      \ar[d, "\mu"']
      \ar[r, "\bigotimes_i\! \eta^{\NRR(J\sslash{/}J^2)}"']
      \&
      \bigotimes_{i=0}^{n+1} \NRR(J_m/J_m^2)
      \ar[dr, "\mu"]
      \&
      \\
      \NRR(J_n/J_n^2)
      \ar[rr, "\eta^{\NRR(J\sslash{/}J^2)}"']
      \&\&
      \NRR(J_m/J_m^2)
    \end{tikzcd}
  \end{gather}

  In light of these computations, to conclude the proof, it is
  sufficient to establish the equality
  \begin{equation}
    \label{eq:wts}\textstyle
    \mu \circop \bigotimes_i \left(H_{m,\eta^{\Delta^1}\!(i)} \circop \eta^{\NRR(I/I^2)}\right)
    = \mu \circop \bigotimes_i \left(\eta^{\NRR(J/J^2)} \circop H_{n,i}\right).
    \tag{\(\star\)}
  \end{equation}
  Note that
  \[
    H_{m,\eta^{\Delta^1}\!(i)} \circop \eta^{\NRR(I/I^2)}
    = \NRR(h_{m,\eta^{\Delta^1}\!(i)} \circop \eta^{I/I^2}),
  \]
  and
  \[
    \eta^{\NRR(J/J^2)} \circop H_{n,i}
    = \NRR(\eta^{J/J^2} \circop h_{n,i}).
  \]
  Rewriting \(K(I/I^2) = k \oplus I/I^2\) and performing binomial
  expansion on the tensor products gives
  \[
    \bigotimes_{i=0}^{n+1} \NRR(I_n/I_n^2)
    \cong k \oplus \left(\bigoplus_{i=0}^{n+1} I_n/I_n^2\right) \oplus T_I
  \]
  where \(T_I\) is the sum of all terms with more than one tensor
  factor of \(I_n/I_n^2\). After doing the same with \(J_m/J_m^2\),
  we expand the left hand side of \cref{eq:wts},
  \[
    \begin{tikzcd}[column sep=huge]
      \bigotimes_{i=0}^{n+1} \NRR(I_n/I_n^2)
      \ar[d, "\bigotimes_i\! \NRR(h_{m,\eta^{\Delta^1}\!(i)} \circop \eta^{I\sslash{/}I^2})" ]
      \ar[r, "\cong", sloped]
      &
      k \oplus \left(\bigoplus_{i=0}^{n+1} I_n/I_n^2\right) \oplus T_I
      \ar[d, "1_k \oplus (\bigoplus_i\! h_{m,\eta^{\Delta^1}\!(i)} \circop \eta^{I\sslash{/}I^2}) \oplus t"]
      \\
      \bigotimes_{j=0}^{n+1} \NRR(J_m/J_m^2)
      \ar[r, "\cong", sloped]
      \ar[d, "\mu"]
      &
      k \oplus \left(\bigoplus_{i=0}^{n+1} J_m/J_m^2\right) \oplus T_J
      \ar[d, "1_k \oplus \sigma \oplus 0"]
      \\
      \NRR(J_m/J_m^2) \ar[r, "\cong"]
      &
      k \oplus J_m/J_m^2 \oplus 0,
    \end{tikzcd}
  \]
  where \(t\) denotes the component corresponding to \(T_I \to T_J\).
  Commutativity of the bottom square follows from the fact that the
  multiplication map restricts trivially to \(T_J\).

  The key observation here is that the composite down the center of
  the right column
  \[
    \bigoplus_{i=0}^{n+1} I_n/I_n^2
    \xto{\bigoplus_i\! h_{m,\eta^{\Delta^1}\!(i)} \circop \eta^{I\sslash{/}I^2}}
    \bigoplus_{i=0}^{n+1} J_m/J_m^2
    \xto{\sigma}
    J_m/J_m^2
  \]
  is precisely \(h_m \circop \eta^{\Cyl_\Mod(I/I^2)}\). This can be seen
  exactly as in diagram \cref{dia:htpy} if performed in the module
  setting.

  Performing the above decomposition for the right hand side of
  \cref{eq:wts} yields
  \[
    \begin{tikzcd}[column sep=huge]
      \bigotimes_{i=0}^{n+1} \NRR(I_n/I_n^2)
      \ar[d, "\bigotimes_i\! \NRR(\eta^{J\sslash{/}J^2} \circop h_{n,i}))" ]
      \ar[r, "\cong", sloped]
      &
      k \oplus \left(\bigoplus_{i=0}^{n+1} I_n/I_n^2\right) \oplus T_I
      \ar[d, "1_k \oplus (\bigoplus_i\! \eta^{J\sslash{/}J^2} \circop h_{n,i}) \oplus t'"]
      \\
      \bigotimes_{j=0}^{n+1} \NRR(J_m/J_m^2)
      \ar[r, "\cong", sloped]
      \ar[d, "\mu"]
      &
      k \oplus \left(\bigoplus_{i=0}^{n+1} J_m/J_m^2\right) \oplus T_J
      \ar[d, "1_k \oplus \sigma \oplus 0"]
      \\
      \NRR(J_m/J_m^2) \ar[r, "\cong"]
      &
      k \oplus J_m/J_m^2 \oplus 0.
    \end{tikzcd}
  \]
  for some \(t'\). Similarly to the previous computation, we see that
  the composite down the center of the right column
  \[
    \bigoplus_{i=0}^{n+1} I_n/I_n^2
    \xto{\bigoplus_i\! \eta^{J\sslash{/}J^2} \circop h_{n,i}}
    \bigoplus_{i=0}^{n+1} J_m/J_m^2
    \xto{\sigma}
    J_m/J_m^2
  \]
  is exactly \(\eta^{J/J^2} \circop h_n\).

  By assumption, \(h\) is a simplicial module map, and hence it
  commutes with simplicial structure maps
  \[h_m \circop \eta^{\Cyl_\Mod(I/I^2)} = \eta^{J/J^2} \circop h_n.\]
  Therefore, the above diagrams only differ by the terms
  \(t,t' \co T_I \to T_J\). Since these maps are equalized by
  multiplication, we establish \cref{eq:wts}, concluding the proof.
\end{proof}

\begin{remark}\label[remark]{rem:proof-with-face}
  To further illuminate the details in the above proof, consider the case of
  the \(i\)\tth\ face map. The face map
  \[
    d_i^{\Cyl_\Ring(\NRR(I/I^2))} \co \bigotimes_{i=0}^{n+1} \NRR(I_n/I_n^2)
    \to \bigotimes_{i=0}^n \NRR(I_{n-1}/I_{n-1}^2)
  \]
  is defined by applying \(d_i^{\NRR(I/I^2)}\) to each component and
  multiplying the \(i\)\tth\ and \((i+1)\)\tth[st]\ component
  \[
    \begin{tikzcd}[column sep = 0em, row sep = 1ex]
      f_0 \ar[maps to, dr] & \otimes & \dots & \otimes
      & f_i \ar[maps to, dr] & \otimes & f_{i+1} \ar[maps to, dl]
      & \otimes & \dots & \otimes & f_{n+1} \ar[maps to, dl]
      \\
      & d_if_0  & \otimes & \dots & \otimes
      & d_if_i \cdot d_if_{i+1}
      & \otimes & \dots & \otimes & d_if_{n+1}.
    \end{tikzcd}
  \]
  Let \(f_i = a_i \oplus x_i \in k \oplus I_n/I_n^2\). For computational
  simplicity, assume \(a_i = 1\) for all \(i\). Then the binomial
  expansion of the above proof gives
  \[
    d_i^{\Cyl_\Ring(\NRR(I/I^2))}(f_0 \otimes \cdots \otimes f_{n+1})
    = 1 \oplus d_ix_0 \oplus \cdots \oplus (d_ix_i + d_ix_{i+1}) \oplus \cdots \oplus d_ix_{n+1} \oplus d_ig
  \]
  where \(g\) consists of higher order terms. Applying
  \(H_{n-1} = \mu \circ \bigotimes_i H_{n-1,i}\) gives
  \begin{align*}
    &1 + h_{n-1,0}(d_ix_0) + \cdots + h_{n-1,i}(d_ix_i + d_ix_{i+1}) + \cdots + h_{n-1,n}(d_ix_{n+1}) + H_{n-1}d_ig \\
    &= 1 + h_{n-1}(d_ix_0 \oplus \cdots \oplus (d_ix_i + d_ix_{i+1}) \oplus \cdots \oplus d_ix_{n+1}) + 0 \\
    &= 1 + h_{n-1}d_i^{\Cyl_\Mod(I/I^2)}(x_0 \oplus \cdots \oplus x_{n+1}) + 0 \\
    &= 1 + d_i^{J/J^2}h_n(x_0 \oplus \cdots \oplus x_{n+1}) + 0 \\
    &= 1 + d_i^{J/J^2}(h_{n,0}x_0 + \cdots + h_{n,n+1}x_{n+1}) + 0 \\
    &= 1 + d_i^{J/J^2}(h_{n,0}x_0 + \cdots + h_{n,n+1}x_{n+1}) + d_i^{J/J^2} H_n g \\
    &= d_i^{\NRR(J/J^2)}H_n(f_0 \otimes \cdots \otimes f_{n+1}).
  \end{align*}
  where we freely remove \(H_{n-1}d_ig\) and insert \(d_i^{J/J^2}H_n g\),
  since factoring through \(\mu\) kills both terms.
\end{remark}

\begin{lemma}\label[lemma]{lem:I-lift-J2}
  Let \(\phi \co I \to J\) be a map in \(\sNRing(k)\), and let
  \(\bar\phi\) denote the induced map on the quotients
  \(I/I^2 \to J/J^2\). Suppose \(I\) is cofibrant. If
  \(\bar\phi \Cylsim[\NRing] 0\), then there exists a lift
  \(\tilde{\phi} \co I \to J^2\) in \(\sNRing(k)\) so that
  \[
    \begin{tikzcd}
      & J^2 \ar[d] \\
      I \ar[r, "\phi"'{name=A}]
      \ar[ur, dashed, "\tilde{\phi}"{name=B}]
      \ar[
        from=B,
        to=2-2,        
        phantom,
        "\simeq"{rotate=-90}
      ]
      & J
    \end{tikzcd}
  \]
  commutes up to homotopy.
\end{lemma}

\begin{proof}
  Let \(H \co \Cyl_\NRing(I/I^2) \to J/J^2\) be a homotopy witnessing
  \(\bar\phi \Cylsim[\NRing] 0\). Since \(I\) is cofibrant, the
  inclusion
  \[i_0 \ceq I \cong I \sqcup 0 \to I \sqcup I \to \Cyl_\NRing(I)\]
  is an acyclic cofibration, and hence the composite
  \[\Cyl_\NRing(I) \to \Cyl_\NRing(I/I^2) \to J/J^2\] %
  lifts to a \(\Cyl_\NRing\)-homotopy \(\widetilde{H}\) as in the following
  diagram:
  \[
    \begin{tikzcd}
      I \ar[r, "\phi"] \ar[d, hook, shift right, "\sim" sloped, "i_0"'] & J \ar[d, two heads] \\
      \Cyl_\NRing(I) \ar[r] \ar[ur, dashed, "\widetilde{H}"] & J/J^2.
    \end{tikzcd}
  \]
  However, by naturality of the component maps, the composite of
  \(i_1\) with \(\Cyl_\NRing(I) \to J/J^2\) factors through zero,
  \[
    \begin{tikzcd}[anchor=south]
      & & J^2 \ar[d] \\
      I \ar[r, "i_1"] \ar[d] \ar[urr, dashed, bend left=20, "\tilde{\phi}"]
      & \Cyl_\NRing(I) \ar[r, "\widetilde{H}"] \ar[d] & J \ar[d, two heads] \\
      I/I^2 \ar[r, "i_1"] \ar[rr, bend right=20, "0"'] & \Cyl_\NRing(I/I^2) \ar[r, "H"] & J/J^2,
    \end{tikzcd}    
  \]
  inducing the map \(\tilde{\phi} \co I \to J^2\).
\end{proof}

We are now ready to prove the main theorem.

\begin{reptheorem}{theorem}{thm:ghost-ring}\label[theorem]{rep:thm:ghost-ring}
  Let \(k\) be a field. Suppose \(S \in \sRing(k)\) is connected and
  \(\pi_i(S) = 0\) for all \(i \ge 2^n\).  If
  \(R = R_0 \to \cdots \to R_n = S\) is a sequence of ghost maps, then
  the composite \(R \to S\) is nullhomotopic.
\end{reptheorem}

\begin{proof}
  Since the ghost property is invariant under cofibrant replacement,
  we assume \(k \into R_0 \into \cdots \into R_n\) is a sequence of
  cofibrations.

  Let \(I\) denote the augmentation ideal of \(R\), \(J\) denote the
  augmentation ideal of \(S\), and \(I_i\) denote the augmentation
  ideal of \(R_i\), so that \(I = I_0\) and \(J = I_n\). By
  hypothesis, each map
  \(\pi_*(I_i/I_i^2) \to \pi_*(I_{i+1}/I_{i+1}^2)\) is
  zero. Furthermore, since \(k\) is a field,
  \((I_i/I_i^2 \to I_{i+1}/I_{i+1}^2) \Cylsim[\Mod] 0\) for each
  \(i\).  By \cref{lem:mod-to-alg} and \cref{lem:I-lift-J2}, there
  exist maps \(I_i \to I_{i+1}^2\) lifting \(I_i \to I_{i+1}\) up to
  homotopy. As these maps are algebra homomorphisms, they are
  compatible with the restrictions
  \[
    \begin{tikzcd}
      I_i^i  \ar[r] \ar[d] & \cdots \ar[r] & I_i^2  \ar[r] \ar[d] & I_i \ar[d] \\
      I_{i+1}^{2^i} \ar[r]   & \cdots \ar[r] & I_{i+1}^4 \ar[r]      & I_{i+1}^2.
    \end{tikzcd}
  \]
  Composing these restrictions
  \(I = I_0 \to \cdots \to I_n^{2^n} = J^{2^n}\) produces a map
  factoring \(I \to J\) up to homotopy. By hypothesis, \(J\) is
  connected and cofibrant, so by \cref{lem:homotopy-powers}, \(\pi_i(J^{2^n}) = 0\)
  for all \(i < 2^n\). Thus \cref{lem:null-inc} implies
  \(J^{2^n} \to J\) is nullhomotopic. Applying the equivalence of
  categories \(\NRR \co \sNRing(k) \to \sRing(k)\),
    \[
    \begin{tikzcd}[anchor=south]
      ~ &
      \NRR(J^{2^n}) \ar[r] \ar[rrd, ""{name=B,above}]  &
      k \ar[rd]  &
      ~ \\
      R \ar[ur] \ar[rrr, ""{name=A, above}] &
      ~ &
      ~ &
      S
      \ar[from=2-2, to=1-2, phantom, "\simeq"{rotate=-90}]
      \ar[from=B, to=1-3, phantom, "\simeq"{rotate=-90}]
    \end{tikzcd}
  \]
  concludes the proof.
\end{proof}

\begin{remark}\label[remark]{rem:ghost-nonfield}
  Let \(f_i \co R_i \to R_{i+1}\) be as in \cref{thm:ghost-ring}. The
  assumption that \(k\) is a field is only necessary to conclude that
  \(\AQ(f_i) \Cylsim[\Mod] 0\). Indeed, \cref{thm:ghost-ring} holds
  for any \(k \in \sRing\) if one replaces the assumption that
  each \(f_i\) is ghost with the stronger assumption that each
  \(\AQ(f_i)\) is nullhomotopic in \(\sMod(k)\).
\end{remark}

%% file: singularities.tex
\section{Applications to Singularities}\label{sec:singularities}
We now use \cref{thm:ghost-ring} to demonstrate how ghost maps detect
singularities. We begin by reviewing the definition of CI-dimension.

\begin{notation}\label[notation]{not:singularities}
  In this section, we adhere to the notational conventions of
  \cref{not:cats}, \cref{not:aq} and \cref{not:frob-koszul}.
\end{notation}

\begin{definition}\label[definition]{def:ci-dim}\cite[1.2]{AvramovGasharovPeeva97}
  Let \(R\) be a local ring. A \emph{quasi-deformation} of \(R\) is a
  diagram \(R \to R' \from Q\), where \(R \to R'\) is a flat local
  extension and \(R' \from Q\) is a surjective local map whose kernel
  is generated by a regular sequence. If \(M\) is an \(R\)-module,
  then the \emph{complete intersection dimension}, abbreviated
  \emph{CI-dimension}, is defined by
  \[
    \cid_R(M)
    \ceq \inf\{
    \pd_Q (R' \otimes_R M)
    - \pd_Q R' \mid R \to R' \from Q \text{ is a quasi-deformation}\}.
  \]
  See \cite{Sather-Wagstaff08} for a discussion of the case where
  \(M\) is not finitely generated.
\end{definition}

\begin{lemma}\label[lemma]{lem:cid-k}\cite[1.3]{AvramovGasharovPeeva97}
  Let \((R,k)\) be a local ring. Then \(R\) is a complete
  intersection if and only if \(\cid_R(k) < \infty\). \qed
\end{lemma}

\begin{lemma}\label[lemma]{lem:pullback-equiv-sring}
  If \(f,g \co  \to S\) is a pair of maps in \(\sRing\) which
  agree in \(\Ho(\sRing)\), then \(S_f \simeq S_g\) as \(R\)-algebras.
\end{lemma}
\begin{proof}
  First, note that \(\sRing\) is left-proper. This follows immediately
  from \cref{rem:pushouts-ring-mod}, since the tensor product of
  modules is left-Quillen. Second, note that all objects in
  \(\sRing\) are fibrant. If \(p \co \widetilde{R} \ontosim R\) is a
  cofibrant replacement, then by \cref{rem:cyl-simp} and
  \cref{lem:pullback-equiv}, \(S_{fp} \simeq S_{gp}\) over
  \(\widetilde{R}\) and hence \(S_f \simeq S_g\) over \(R\) by
  \cref{rem:left-proper}.
\end{proof}

\begin{lemma}\label[lemma]{lem:bounded-tor}
  Let \(Q \to R\) be a map in \(\Ring\) with finite flat dimension.
  Suppose \(N\) is a homologically bounded complex of \(Q\)-modules
  and \(K\) is a homologically bounded complex of \(R\)-modules. If
  \(K\) has finite flat dimension over \(R\), then \(\Tor^Q_*(N,K)\)
  is bounded.
\end{lemma}
\begin{proof}
  Suppose first that \(N\) is a \(Q\)-module. Note that
  \(N \lotimes_Q K \simeq (N \lotimes_Q R) \lotimes_R K\) and consider
  the associated Tor spectral sequence \citestacks{0662}
  \[
    E^2_{p,q} = \Tor^R_p(\Tor^Q_q(N, R), K) \Rightarrow \Tor^Q_{p + q}(N,K).
  \] %
  By hypothesis, \(R\) has finite flat dimension over \(Q\), so
  \(E^2_{p,q} = 0\) for all \(|q| \gg 0\). Likewise, \(K\) has
  finite flat dimension over \(R\), so \(E^2_{p,q} = 0\) for all
  \(|p| \gg 0\). Thus \(E\) is a bounded spectral sequence, so
  \(\Tor^Q_*(N, K)\) must be bounded.

  The general case follows by first applying essentially the same
  spectral sequence argument with the replacements
  \((Q,N,R,K) \rightsquigarrow (Q,N,Q,R)\) and using the boundedness
  hypothesis on \(N\) to conclude that \(\Tor^Q_*(N,R)\) is bounded,
  and then proceeding as before.
\end{proof}

Let \((R,k)\) be a local ring and \(\phi \co R \to R\) be a local
endomorphism. For \(n > 0\), write \(K^R_n\) for \(K^R\) viewed as an
\(R\)-algebra via \(R \xto{\phi^n} R \to K^R\), and let
\(K^{R,\tau}_n\) denote \(K^R\) viewed as an \(R\)-algebra via the map
\(R \to k \xto{\phi^n} k \to K^R\).

\begin{lemma}\label[lemma]{lem:ghost-koszul-trivial}
  Let \((R,k)\) be a noetherian local ring and \(\phi\) a local
  endomorphism such that \(K^\phi\) is ghost. For all \(n \gg 0\),
  \(K^R_n \simeq K^{R,\tau}_n\) as \(R\)-algebras.
\end{lemma}

\begin{proof}
  Let \(\mfm\) be the maximal ideal of \(R\) and write \(K = K^R\) for
  the Koszul complex. Suppose \(\mfm\) is minimally generated by \(d\)
  elements and fix \(n > \log_2(d)\). Iterating the construction of
  \(K^\phi\) in \ref{con:Koszul-cohen} gives a sequence of ghost maps in
  \(\sRing(k_{\phi^n})\)
  \[ K_{0,n} \to \cdots \to K_{n,n}, \]
  where \(K_{i,n} = k_{\phi^n} \lotimes_{k_{\phi^i}} K_i\). Since
  \(K_{n,n} \simeq K_n\) has the same underlying simplicial ring as \(K\),
  we observe (\cref{rem:Koszul-is-Koszul})
  \[
    \begin{cases}
      \pi_0(K_{n,n}) = k_{\phi^n}  \\
      \pi_i(K_{n,n}) = 0 & i > d.
    \end{cases}
  \]
  By \cref{thm:ghost-ring}, the composite \(K_{0,n} \to K_{n,n}\) is
  nullhomotopic in \(\sRing(k_{\phi^n})\), so restriction along
  \(\bbZ \to k_{\phi^n}\) yields the following commutative diagram in \(\Ho(\sRing)\)
  \[
    \begin{tikzcd}
      R \ar[rr, "\phi^n"] \ar[d]  && R_{\phi^n} \ar[d] \\
      K_0 \ar[r] \ar[d] & K_{0,n} \ar[r, ""{name=A}] \ar[d] &  K_n \\
      k \ar[r, "\phi^n"] & k_{\phi^n}. \ar[ur, ""{name=B}]
      \ar[from=2-2, to=B, phantom, "\simeq"{rotate=-90}]
    \end{tikzcd}
  \]
  Therefore, by \cref{lem:pullback-equiv-sring}, we see that
  \(K_n \simeq K_n^\tau\) as \(R\)-algebras.
\end{proof}

\begin{remark}\label[remark]{rem:tor-factorization}
  Let \(\phi \co R \to R\) be a local ring endomorphism. For any
  simplicial \(R\)-module \(M\) we define \(K^M\) to be the simplicial
  \(K^R\)-module \(K^M \ceq K^R \lotimes_R M\).

  Suppose there exists an \(n\) as in \cref{lem:ghost-koszul-trivial}.
  Likewise, define the simplicial \(R\)-modules \(K^M_n\) and
  \(K^{M,\tau}_n\) via the relevant restrictions. Since
  \(K^R_n \simeq K^{R,\tau}_n\) over \(R\), we obtain a homotopy
  equivalence of simplicial \(R\)-modules
  \[K^M_n \simeq K^{M,\tau}_n.\]
  As a consequence, for any simplicial \(R\)-module \(N\), we obtain
  the following \(\Tor\) factorization
  \[
    \Tor^R_*(N,K^M_n)
    \cong \Tor^R_*(N, K_n^{M,\tau})
    \cong \pi_*\left(N \lotimes_R k \otimes_k K_n^{M,\tau} \right) 
    \cong  \Tor^R_*(N,k) \otimes_k \pi_*(K_n^{M,\tau}).
  \]
  Further, since the underlying simplicial rings agree, we calculate
  \[
    \pi_*(K_n^{M,\tau})
    \cong k_{\phi^n} \otimes_k \pi_*(K^M)
    \cong  k_{\phi^n} \otimes_k \Tor^R_*(K^R,M).
  \]
  We note that this factorization extends to any (possibly unbounded)
  chain complex \(M\) over \(R\) by passing the equivalence
  \(K^R_n \simeq K^{R,\tau}_n\) through the Dold-Kan theorem (see
  \cref{rem:AQ-complex}).
\end{remark}

\begin{remark}\label[remark]{rem:extremal}
  Again, consider the setting of \cref{lem:ghost-koszul-trivial} and
  \cref{{rem:tor-factorization}}.  If \(\pi_*(M)\) is bounded and
  degree-wise finite over \(R\), then \(\pi_*(K_n^{M,\tau})\) is a
  bounded degree-wise finite graded \(k_{\phi^n}\)-module by
  \cref{rem:Koszul-is-Koszul}.  In particular, applying \(N = k\) to
  \cref{rem:tor-factorization} shows that the asymptotic growth of the
  of \(\rank_k\Tor^R_i(k,k)\) and
  \(\rank_{k_\phi^n}\Tor^R_i(k,K^M_n)\) differ only in multiplicity
  (more precisely, the corresponding Hilbert series only differ by a
  polynomial factor given by the Hilbert series for \(\pi_*(K^M)\)).

  The asymptotic growth invariants curvature and complexity
  \cite{Avramov96, AvramovIyengarMiller06,
    AvramovHochsterIyengarYao12}, denoted \(\cx\) and \(\curv\)
  respectively, are unaffected by multiplicity (more precisely, they
  are unaffected by non-zero polynomial factors in the corresponding
  Hilbert series \cite[3.1]{AvramovHochsterIyengarYao12}.
  Hence, if \(M \not\simeq 0\), we see
  \[
    \curv_R(k) = \curv_{\phi^n}(M)
    \quad\text{and}\quad
    \cx_R(k) = \cx_{\phi^n}(M),
  \]
  where the right side invariants above are defined relative to the
  Koszul complex (see \cite[7.2.1]{AvramovIyengarMiller06}).
  
  Since \(k\) provides a global upper bound for these invariants, a
  module or complex obtaining these bounds is said to be
  \emph{extremal}. We restate the conclusion of this remark in
  \cref{pro:extremal} below.
\end{remark}

\begin{proposition}\label[proposition]{pro:extremal}
  Let \(R\) be a noetherian local ring and \(\phi\) a local ring
  endomorphism such that \(K^\phi\) is ghost. If \(M \not\simeq 0\) is
  a homologically bounded and finite complex of \(R\)-modules, then
  \(M\) is extremal relative to \(\phi^n\) for \(n \gg 0\). \qed
\end{proposition}

We now move to proving the main theorem of this section. While the
proof may be phrased in terms of extremal modules, we opt for a more
elementary and direct reduction to the Auslander-Buchsbaum-Serre
theorem.

\begin{theorem}\label[theorem]{thm:kunzish}
  Let \(R\) be a noetherian local ring and \(\phi\) a local
  endomorphism such that \(K^\phi\) is ghost.
  \begin{enumerate}[label=\roman*)]
      \item\label{enum:ffd-reg} If \(\phi\) has finite flat dimension, then \(R\) is regular.
      \item\label{enum:fcid-ci} If \(\phi^n\) has finite CI-dimension
    for some \(n \gg 0\), then \(R\) is a complete intersection.
  \end{enumerate}
\end{theorem}
\begin{proof} Write \(K = K^R\). In both cases we apply
  \cref{lem:ghost-koszul-trivial} to obtain \(K_n \simeq K_n^\tau\).
  
  (i) Suppose \(\phi\) has finite flat dimension. Then \(\phi^n\) has
  finite flat dimension by inductive application of
  \cref{lem:bounded-tor}, and applying the lemma again implies
  \(\Tor^R_*(k, K_n)\) is bounded. By \cref{rem:tor-factorization}
  \[
    \Tor^R_*(k,K_n) \cong \Tor^R_*(k,k) \otimes_k k_{\phi^n} \otimes_k \pi_*(K)
  \]
  We conclude that \(\Tor^R_*(k,k)\) is bounded, and hence \(R\) is
  regular by the Auslander-Buchsbaum-Serre theorem.

  (ii) Let \(R_n\) denote \(R\) as an \(R\)-algebra via
  \(R \xto{\phi^n} R\). Suppose \(\phi^n\) has finite
  CI-dimension. Then there exists a flat local extension \(R \to R'\)
  with residue field \(k'\) and a deformation \(Q \to R'\) such that
  \(R' \otimes_R R_n\) has finite flat dimension over
  \(Q\). Since \(K_n\) is a bounded complex of free
  \(R_n\)-modules,
  \(R' \otimes_R K_n \cong (R' \otimes_R R_n)
  \otimes_{R_n} K_n\) is a bounded complex of free
  \(R' \otimes_R R_n\)-modules. Thus, applying
  \cref{lem:bounded-tor}, we see that
  \(\Tor^Q_*\left(k', R' \otimes_R K_n\right)\) is bounded.
  Analogously to the proof of (i),
  \[
    \Tor^Q_*\left(k', R' \otimes_R K_n\right)
    \cong \Tor^Q_*\left(k', R' \otimes_R k\right) \otimes_k k_{\phi^n} \otimes_k \pi_*(K).
  \]
  Because \(R' \otimes_R k\) is finitely generated over \(Q\), the
  above calculation implies \(R' \otimes_R k\) has finite
  projective dimension over \(Q\), and hence \(k\) has finite
  CI-dimension over \(R\). The result follows by \cref{lem:cid-k}.
\end{proof}

\begin{corollary}[Kunz-Rodicio]\label[corollary]{cor:kunz}\cite{Kunz69,Rodicio88}
  Suppose \(R\) is a noetherian ring of prime characteristic and let
  \(F\) denote the Frobenius endomorphism. The following are equivalent.
  \begin{itemize}
      \item \(R\) is regular
      \item \(F\) is flat
      \item \(F\) has finite flat dimension
  \end{itemize}
\end{corollary}
\begin{proof}
  Let \(p = \chr R\). Since flatness and regularity are local
  properties and unaffected by completion, we may assume that \(R\) is
  a complete local ring with residue field \(k\). Fix a minimal Cohen
  presentation \(Q \to R\) where
  \(Q = k \ldsq x_1, \ldots, x_d \rdsq\).
  
  Suppose \(R\) is regular. Then \(R = Q\) and we directly compute
  that \(Q_F\) is the free \(Q\)-module on the basis
  \(\{ x_1^{n_1} \cdots x_d^{n_d} \mid 0 \le n_i < p\}\).

  Now suppose \(F\) has finite flat dimension. By
  \cref{rem:frob-is-koszul}, \(K^F\) is given by
  \(k_F \lotimes_{Q_F} F_{\hat{R}|Q}\) and hence \cref{pro:frob-ghost}
  states that \(K^F\) is ghost, so \cref{thm:kunzish} implies \(R\) is
  regular.
\end{proof}

Using part (ii) of \cref{thm:kunzish} we may establish an analogous
result for complete intersection rings.

\begin{corollary}\label[corollary]{cor:kunz-ci}
  Suppose \(R\) is a noetherian ring of prime characteristic and let
  \(F\) denote the Frobenius endomorphism. If \(F^n\) has finite
  CI-dimension for some \(n \gg 0\), then \(R\) is a locally complete
  intersection ring. \qed
\end{corollary}

\begin{remark}\label[remark]{rem:contracting}
  Let \(\phi \co R \to R\) be contracting, as in
  \cref{def:contracting}. In \cite{AvramovIyengarMiller06,
    AvramovHochsterIyengarYao12}, contracting endomorphisms are
  studied in a similar manner to our investigation of ghost maps. The
  ghost condition can be thought of as a derived contracting
  condition. Compare, for example, \cref{thm:kunzish} part (i) with
  \cite[Thm 13.3]{AvramovIyengarMiller06}, whose statement implies
  \emph{If \(\phi\) is contracting, then \(\phi\) has finite flat
    dimension if and only if \(R\) is regular}. Indeed, the
  connection is already apparent when the ring map is itself ghost by
  \cref{cor:contracting}.  %
  
  This result is taken further by \cite[Thm
    1.1]{AvramovHochsterIyengarYao12}, weakening the finite flat
  dimension hypothesis from \(R\) to any finite \(R\)-module
  \(M\). The proof follows from \cite[Thm
    5.1]{AvramovHochsterIyengarYao12}, establishing
  \(\curv_{\phi^n}(M) = \curv_R(k)\), parallel to
  \cref{pro:extremal}. The \(\Tor\) factorization from which we derive
  this proposition in \cref{rem:tor-factorization} has an analogue in
  the contracting case as well \cite[Prop
    4.3(2)]{AvramovHochsterIyengarYao12}. Both papers also contain
  extensions to the complete intersection case similar to
  \cref{thm:kunzish} part (ii), see in \cite[Thm 13.4, Rem
    13.5]{AvramovIyengarMiller06} and \cite[Cor
    5.3(2)]{AvramovHochsterIyengarYao12}.
\end{remark}

\begin{remark}\label[remark]{rem:majadas}
  Given a local ring homomorphism \(\phi \co R \to R\) as above,
  Majadas \cite{Majadas16} defines \(\phi\) to have the
  \(h_2\)-\emph{vanishing} property if the induced map
  \[\AQ_2(k_\phi \dslash R) \to \AQ_2(k_\phi \dslash R_\phi)\]
  is zero. Using \cref{rem:kunzish-ghost}, the induced map on homotopy
  cofiber sequences
  \[
    \begin{tikzcd}
      \hat{R} \dslash \Lambda \ar[r] \ar[d] &
      k_\phi \dslash \Lambda \ar[r] \ar[d] &
      k_\phi \dslash \hat{R} \ar[d] \\
      \hat{R}_\phi \dslash \Lambda \ar[r] &
      k_\phi \dslash \Lambda \ar[r] &
      k_\phi \dslash \hat{R}_\phi,
    \end{tikzcd}
  \]
  and the calculation \(\AQ_2(k_\phi \dslash \Lambda) = 0\), one
  deduces that the ghost property for \(K^\phi\) implies the
  \(h_2\)-vanishing property for \(\phi\). With this, we see
  \cref{thm:kunzish} (i) follows from \cite[Prop 5]{Majadas16}.

  Blanco and Majadas \cite{BlancoMajadas98} prove a result similar to
  \cref{cor:kunz-ci} in terms of the CI-dimension of a Cohen
  factorization. In the \(F\)-finite case, one recovers
  \cref{cor:kunz-ci}. Denote this invariant \(\cid(\phi)\).
  Generalizing these ideas, Majadas shows \cite[Prop 8]{Majadas16}
  that if \(\phi\) satisfies \(h_2\)-vanishing and
  \(\cid(\phi) < \infty\), then \(R\) is a complete intersection
  ring. As far as the author is aware, these results are separated by
  two open problems. The first is whether \(\cid(\phi)\) and
  \(\cid_R(R_\phi)\) are simultaneously finite (see
  \cite[2.15]{Sather-Wagstaff08} and \cite{Sharif24}). The second is
  whether finiteness of \(\cid_R(R_\phi)\) implies finiteness of
  \(\cid_R(R_{\phi^n})\).
\end{remark}